\documentclass[draft]{amsproc}
\usepackage{amssymb}
\usepackage[english]{babel}
\usepackage{mathrsfs}
\usepackage{euscript}
\usepackage{mathrsfs,color}
\usepackage{tikz}
\usetikzlibrary{positioning}

\newtheorem{thm}{Theorem}[section]
\newtheorem{cor}[thm]{Corollary}
\newtheorem{lem}[thm]{Lemma}
\newtheorem{pro}[thm]{Proposition}
\theoremstyle{definition}
\newtheorem{exa}[thm]{Example}

\numberwithin{equation}{section}

\newcommand*{\cf}{C_{\phi}}

\newcommand*{\D}{\mathrm{d}}
\newcommand*{\hh}{\mathcal H}
\newcommand*{\hn}[1]{\mathsf{h}_{\phi^{#1}}}
\newcommand*{\h}[1]{\mathsf{h}_{#1}}

\newcommand*{\lambdab}{{\boldsymbol\lambda}}

\newcommand*{\B}[1][\hh]{\boldsymbol{B}(#1)}
\newcommand*{\ascr}{\mathscr{A}}
\newcommand*{\bscr}{\mathscr{B}}
\newcommand*{\gammab}{\boldsymbol \gamma}

\newcommand*\nbb{\mathbb{N}}
\newcommand*\rbb{\mathbb{R}}
\newcommand*\C{\mathbb{C}}
\newcommand*\zbb{\mathbb{Z}}

\newcommand*\dom{\mathcal{D}}
\newcommand*\Rng{\mathcal{R}}
\newcommand*\Ker{\mathcal{N}}

\def\RngInf{\operatorname{\Rng^\infty}}

\begin{document}
	
	\title[$m$--isometric composition operators]{$m$--isometric composition operators on directed graphs with one circuit}
	\author{Zenon Jan Jab\l o\'nski}
	\author{Jakub Ko\'smider}
	\address{Instytut Matematyki, Uniwersytet Jagiello\'nski, ul. \L{}ojasiewicza 6, PL-30348 Kra\-k\'ow, Poland}
	\email{zenon.jablonski@im.uj.edu.pl; jakub.kosmider@im.uj.edu.pl}
	\thanks{}
	\dedicatory{}
	
	
	\begin{abstract}
		The aim of this paper is to investigate
		$m$--isometric composition operators on 
		directed graphs with one circuit. We establish a
		characterization of $m$--isometries and prove that complete
		hyperexpansiveness coincides with $2$--isometricity
		within this class. We
		discuss the $m$--isometric completion problem for
		unilateral weighted shifts and for composition
		operators on directed graphs with one circuit.
		The paper is concluded with an affirmative
		solution of the Cauchy dual subnormality problem
		in the subclass with circuit containing one
		element.
	\end{abstract}
	
	\maketitle
	
	\section{Introduction}
	
	The notion of $m$--isometric operator was introduced by Agler in \cite{agler-1990} and initially investigated by Agler and Stankus in \cite{A-S-1995}.
	Recently, there have been published many papers devoted to problems related to $m$--isometric operators (see \cite{A-L-2016,B-S-2018,B-B-M-P-2019,B-M-N-2016,B-M-N-2010,gu-2017,MC-R-2015}).
	In particular, much attention was paid to $m$--isometric unilateral and bilateral weighted shifts.
	Berm\'udez et al. characterized $m$--isometric unilateral weighted shifts in \cite[Theorem~3.4]{B-M-N-2010}.
	Later on, Abdullah and Le provided convenient characterizations of $m$--isometric unilateral and bilateral
	weighted shifts that bind their sequences of weights with certain real polynomials in one variable of degree at most $m-1$ (see \cite[Theorem~2.1, Theorem~5.2]{A-L-2016}).
	
	Lately,  more general
	classes of operators such as $(m,p)$--isometries
	(see \cite{B-M-N-2016,gu-2017}) have been investigated. Both papers \cite{B-M-N-2016} and \cite{gu-2017}
	deal with composition operators on
	$\ell_p$ for $p\ge 1$. Gu provided a
	characterization of $(m,p)$--isometric
	composition operators defined on $\ell_p$ that
	ties cardinality of preimages
	$\phi^{-n}(j)$ for $j\in\nbb$ with certain polynomials of
	degree at most $m-1$ (see
	\cite[Theorem~2.9]{gu-2017}). The operators
	considered in the above papers can be seen as
	composition operators on discrete measure
	spaces with counting measures. As opposed to
	them, we are focused on composition operators
	on discrete measure spaces with arbitrary
	discrete measures (see Section 2).
	
	The paper is devoted to investigation of classes of $m$--isometric operators, which are more rich in examples than the classes of unilateral and bilateral weighted shifts.
	We focus on the class of composition operators on directed graph with one circuit.
	It was used in \cite[Section 3]{B-J-J-S-2017} in connection with the study of unbounded subnormality.
	Its weighted counterpart appeared much earlier in \cite{carlson-1990}.
	Moreover, these composition operators have been very popular lately.
	For example, Anand, Chavan and Trivedi utilized weighted composition operators of this type to provide an interesting example of analytic $3$--isometric cyclic operator without the wandering subspace property (see \cite[Example 3.1]{A-C-T-2018}).
	Although, the operators belonging to this class can be, in general, unbounded and densely defined, we restrict ourselves to their subclass containing only bounded operators.
	
	As shown in the paper, the class of $m$--isometric composition operators on directed graphs with one circuit is rich in strict $m$--isometries for $m\ge 2$.
	Surprisingly, it turns out that the only completely hyperexpansive composition operators in this class are $2$--isometric (see Corollary~\ref{cor:complete-hyperex}).
	This phenomenon is a consequence of the fact that, in particular, $2$--expansiveness implies $2$--isometricity within this class.
	It is worth to mention that $3$--isometries have been very popular recently (see e.g. \cite{A-C-T-2018,B-B-M-P-2019,MC-R-2015}).
	Except for already mentioned results in \cite{A-C-T-2018}, these operators were, for instance, utilized by Berm\'udez et al. as examples to shed some light on relations of certain ergodic and dynamical properties of $m$--isometric operators (see \cite{B-B-M-P-2019} for more details).
	That is why we provide an explicit description of $3$--isometric composition operators on directed graphs with one circuit that involves real polynomials in one variable (see Theorem~\ref{th:characterization_3_isom}).
	
	The paper is organized as follows.
	Section 2 begins with setting up preliminary notation and definitions.
	Later on, discrete measure spaces and composition operators on directed graphs with one circuit are introduced (see \eqref{def:composition-operator}).
	Next several technical results required in the further proofs are given.
	Theorem \ref{th:m-iso-characterization}, which is the main result of this section, characterizes $m$--isometric composition operators on directed graphs with one circuit.
	As a consequence, Corollary~\ref{cor:m-iso-characterization-one-element} provides a characterization of $m$--isometric operators in the subclass of these operators with the circuit containing only one element.
	Finally, Corollary~\ref{cor:complete-hyperex} shows that the notion of complete hyperexpansiveness coincides with the notion of $2$--isometricity in this class.
	
	In Section 3 we state and investigate the $m$--isometric completion problem for unilateral weighted shifts and composition operators on directed graphs with one circuit.
	At the beginning we provide complete description of existence of a solution of the $m$--isometric completion problem for unilateral weighted shifts (see Proposition~\ref{pro:completion-shifts}).
	Proposition~\ref{pro:completion-composition}, whose proof essentially relies on the properties of circulant matrices, shows that, under some additional conditions on the length of the circuit, there exists a solution of the $m$--isometric completion problem, if measure of elements located on branches of $\phi$ is a priori given by polynomials of degree at most $m-1$.
	Using this fact we establish a characterization of $2$--isometric and $3$--isometric composition operators on directed graphs with one circuit having arbitrary number of elements in the circuit (see Theorem~\ref{th:characterization_3_isom}).
	This result is used in the proof of Proposition~\ref{pro:completion-circuit} which characterizes existence of a solution of another $2$--isometric completion problem, in which we assume that measure of elements located in the circuit is given.
	
	Section 4 is devoted to the study of subnormality of the Cauchy dual operators of $2$--isometric composition operators on directed graphs with one circuit.
	We begin with a characterization of analyticity of composition operators on directed graphs with one circuit (see Proposition~\ref{prop:analiticity}) and prove that $m$--isometric operators within this class are analytic (see Corollary~\ref{cor:analytic}).
	Theorem~\ref{th:subnormal} gives the affirmative answer to the Cauchy dual subnormality problem for $2$--isometric operators with circuit having one element.
	We conclude this section with two results related to $2$--isometric composition operators on directed graphs with one circuit, namely, we state a characterization of $\Delta_{\cf}$--regularity of these operators (see Proposition~\ref{pro:regular}) and show that they never satisfy the kernel condition (see Theorem~\ref{th:kernel-condition}).

	\section{Characterization of $m$--isometries}
	In this paper we use the following notation. 
	Symbols $\nbb$, $\zbb_+$, $\zbb$ and $\rbb_+$ stand for
	the sets of positive integers, nonnegative
	integers, integers and nonnegative reals, respectively. The
	fields of real and complex numbers are denoted by
	$\rbb$ and $\C$, respectively. For $i,j \in
	\zbb\sqcup\{\infty\}$ set $J_{[i,j]} =
	\{k\in\zbb\colon i\le k\le j\}$.
	We adhere to the conventions that
	$\sum_{k\in\emptyset} \alpha_k=0$ and
	$\prod_{k\in\emptyset} \alpha_k=1$.
	
	Define a linear transformation
	$\triangle\colon\rbb^{\zbb_+}\to \rbb^{\zbb_+}$
	by
	\begin{align*}
	(\triangle \gammab)_n=\gamma_{n+1}-\gamma_n,\quad
	n\in \zbb_+,\,
	\gammab=\{\gamma_n\}_{n=0}^\infty\in\rbb^{\zbb_+}.
	\end{align*}
	Relations $\triangle^0\gammab=\gammab$ and
	$\triangle^{j+1}\gammab=\triangle\triangle^{j}\gammab$
	for $j\in\zbb_+$ inductively define $\triangle^j$
	for all $j\in\zbb_+$. Using Newton's binomial
	formula, we can easily prove that
	\begin{align}
	\label{Newt1}
	(\triangle^m \gammab)_n = (-1)^m \sum_{k=0}^m (-1)^k \binom{m}{k} \gamma_{n+k}, \quad m,n\in\zbb_+, \, \gammab=\{\gamma_n\}_{n=0}^{\infty} \in \rbb^{\zbb_+}.
	\end{align}
	Let $\{\gamma_n\}_{n=0}^{\infty}\in\rbb^{\zbb_+}$. We say that $\gamma_n$ is {\em a polynomial
		in $n$} ({\em of degree} $k$) if there exists a
	polynomial $p$ in one indeterminate $x$ with real
	coefficients (of degree $k$) such that
	$p(n)=\gamma_n$ for all $n\in \zbb_+$. As a
	consequence of the Fundamental Theorem~of Algebra
	we see that, if exists, such $p$ is unique.
	
	The following well-known lemma plays a key
	role in our considerations (for more details see
	\cite[Section 2]{J-J-S-2019}).
	
	\begin{lem}
		\label{lem:polynomial}
		If $m\in\nbb$ and
		$\gammab=\{\gamma_n\}_{n=0}^\infty\in
		\rbb^{\zbb_+}$, then the following conditions are
		equivalent{\em :}
		\begin{enumerate}
			\item $\triangle^m \gammab = 0$,
			\item $\sum_{k=0}^m (-1)^k\binom mk \gamma_{n+k}=0$ for $n\in\zbb_+$,
			\item $\gamma_n$ is a polynomial in $n$ of degree at most $m-1$.
		\end{enumerate}
	\end{lem}
	
	\begin{cor}
		\label{cor:delta-polynomial-constant} If
		$m\in\nbb$, $p$ is a polynomial in one
		indeterminate $x$ with coefficients in $\rbb$ of
		the form $p(x)=a_{m-1} x^{m-1}+\ldots+a_0$ and
		$\gammab=\{p(n)\}_{n=0}^\infty\in\rbb^{\zbb_+}$,
		then $(\triangle^{m-1}\gammab)_n=(m-1)!
		\,a_{m-1}$ for all $n\in \zbb_+.$ In particular,
		if $p(n)\ge 0$ for all $n\in\zbb_+$, then
		$\triangle^{m-1}\gammab\subseteq \rbb_+$.
	\end{cor}
	
	Let $\hh$ be a nonzero complex Hilbert space. A
	linear map $T\colon\dom(T)\to\hh$, where
	$\dom(T)$ is a subspace of $\hh$, is called an
	\textit{operator} in $\hh$. By $\B$ we denote the
	$\mathcal{C}^*$-algebra of all bounded linear
	operators defined on $\hh$ and by $I$ the
	identity operator on $\hh$. For $T\in\B$ let
	$\Ker(T)$ and $\Rng (T)$ stand for the kernel and
	the range of operator $T$, respectively. If $T$
	is positive, then by $T^\frac{1}{2}$ we denote
	its positive square root. As usual, we write
	$T^*$ for the adjoint of $T$.
	
	For $T\in\B$ we define
	\begin{align*}
	\bscr_m(T) = \sum_{k=0}^m (-1)^k \binom{m}{k} {T^*}^kT^k, \quad m\in\zbb_+.
	\end{align*}
	Given $m\in\nbb$, we say that $T$ is an {\em
		$m$--isometry} if $\bscr_m(T)=0$. Following
	\cite{Bo-Ja}, we say that an $m$--isometry $T$ is
	{\em strict} if $T$ is an isometry or
	$m\ge 2$ and $T$ is not an $(m-1)$--isometry. 
	An operator $T$ is
	said to an {\em $m$--expansive} if $\bscr_m(T)\le
	0$ and {\em completely hyperexpansive}, if
	$\bscr_n(T)\le 0$ for every $n\in\nbb$.
	
	Let us recall an important property of
	$\bscr_m(T)$ (cf.\
	\cite[Lemma~1(a)]{richter-1988}).
	
	\begin{pro}{\cite[Theorem~2.5]{gu-2015}}
		\label{th:alternating-expansiveness}
		If $T\in\B$, $m\ge2$ is an integer and $(-1)^m\bscr_m(T)\le 0$, then $(-1)^m\bscr_{m-1}(T)\le 0$.
	\end{pro}
	
	If $T\in\B$ and $f\in \hh$, then we define the
	sequence $\gammab_{T,f}=
	\{(\gammab_{T,f})_n\}_{n=0}^{\infty}$ by
	\begin{align*}
	(\gammab_{T,f})_n=\|T^nf\|^2, \quad n\in \zbb_+.
	\end{align*}
	Observe that
	\begin{align}    \label{3007201}
	\langle\bscr_m(T)f,f\rangle=\sum_{k=0}^m (-1)^k
	\binom{m}{k} (\gammab_{T,f})_k, \quad f\in\hh,
	\;m\in \zbb_+.
	\end{align}
	The following fact is a consequence of
	\eqref{3007201}, \eqref{Newt1} and
	Lemma~\ref{lem:polynomial}.
	\begin{pro}  \label{3007202}
		If $T\in\B$ and $m\in\nbb$, then the following
		are equivalent{\em:}
		\begin{enumerate}
			\item[(i)] $T$ is $m$--isometric,
			\item[(ii)] $\sum_{k=0}^m (-1)^k\binom mk (\gammab_{T,f})_{k}=0$ for every $f\in\hh$,
			\item[(iii)] $(\gammab_{T,f})_n$
			is a polynomial in $n$ of degree at most $m-1$
			for every $f\in\hh$.
		\end{enumerate}
	\end{pro}

	For the reader's convenience we
	recall the classification of all directed graphs
	induced by self-maps having only one
	vertex with degree greater than one and all
	other vertices having degrees equal to one. It
	turns out that there are only two possible cases
	(see \cite[Theorem~3.2.1]{B-J-J-S-2017}). The first case is related to the directed trees with
	one branching vertex and the
	second one is connected with the composition
	operators on directed graphs with one
	circuit (see
	Figure~\ref{fig:composition-operator} below).
	From now on we focus on the latter class of operators. 
	
	Let us
	now recall the definition and fundamental facts
	regarding composition operators. By a {\em discrete
		measure space} we mean a measure space
	$(X,\ascr,\mu)$, where $X$ is a countably
	infinite set, $\ascr$ is the $\sigma$-algebra of
	all subsets of $X$ and $\mu$ is a positive
	measure on $\ascr$ such that
	$\mu(x):=\mu(\{x\})\in(0,\infty)$ for all $x\in X$. 
	Suppose now that $(X,\ascr,\mu)$ is
	a measure space and a mapping
	$\phi\colon X\to X$ is measurable, i.e.\
	$\phi^{-1}(\ascr)\subseteq\ascr$. Denote by
	$\mu\circ\phi^{-1}$ the measure on $\ascr$ given
	by
	\begin{align*}
	\mu\circ\phi^{-1}(\varDelta)=\mu(\phi^{-1}(\varDelta)),
	\quad \varDelta \in \ascr.
	\end{align*}
	We say that $\phi$ is {\em nonsingular}, if
	$\mu\circ\phi^{-1}$ is absolutely continuous with
	respect to $\mu$. If $\phi$ is nonsingular, then
	the operator of the form $\cf f:=f\circ\phi$
	acting in $L^2(\mu):=L^2(X,\ascr,\mu)$ with the
	domain
	\begin{align*}
	\dom(\cf):=\{f\in L^2(\mu): f\circ\phi\in L^2(\mu)\}
	\end{align*}
	is well-defined and it is called {\it composition
		operator}. 
	If $(X,\ascr,\mu)$ is a discrete measure space, then $\phi$ is automatically
	nonsingular. 
	Suppose now that $(X,\ascr,\mu)$ is
	a discrete measure space.
	By the Radon-Nikodym theorem, there
	exists a unique $\ascr$-measurable function
	$\hn{} \colon X \to [0,\infty]$ such that
	\begin{align*}
	\mu\circ \phi^{-1}(\varDelta) = \int_\varDelta
	\hn{}\D \mu, \quad \varDelta \in \ascr.
	\end{align*}
	Recall that
	\begin{align}
	\label{eq:boundedness-condition}
	\text{ $\cf$ is bounded if and only if $\hn{}\in L^\infty(\mu)$;}
	\end{align}
	if this is the case, then $\|\cf\|^2 = \|\hn{}\|_{L^\infty(\mu)}$ (see \cite[Theorem~1]{nordgren-1978}).
	If $n \in \nbb$, then by $\phi^n$ we denote the $n$-fold composition of $\phi$ with itself; $\phi^0$ is the identity mapping.
	It is easily seen that if $\phi$ is nonsingular and $n \in \zbb_+$, then $\phi^n$ is also nonsingular, so $\hn{n}$ makes sense; in particular $\hn{0}\equiv 1$ and $\hn{1}=\hn{}$.
	Note that if $\cf\in\B[L^2(\mu)]$, then
	\begin{align}
	\label{aa}
	\|\cf^nf\|^2=\int_X|f\circ\phi^n|^2\mathrm{d}\mu=\int_X \hn{n}|f|^2\mathrm{d}\mu, \quad f\in L^2(\mu),\; n\in\zbb_+.
	\end{align}
	Combining this, Proposition~\ref{th:alternating-expansiveness} and \cite[Lemma~2.1]{jablonski-2003} we get the following fact.
	\begin{align}
	\label{eq:newton-symbol-inequality}
	\begin{minipage}{75ex}
	If $\cf\in\B[L^2(\mu)]$, $m\ge2$ is an integer and $(-1)^m\sum_{k=0}^m (-1)^k \binom{m}{k} \hn{k}\le 0$, then $(-1)^m\sum_{k=0}^{m-1} (-1)^k \binom{m-1}{k} \hn{k}\le 0$.
	\end{minipage}
	\end{align}
	
	The next result, which is a direct consequence of \cite[Lemma~2.3(ii)]{jablonski-2003}, Lemma~\ref{lem:polynomial} and \eqref{aa}, provides a characterization of bounded $m$--isometric composition operators on discrete measure spaces.
	
	\begin{pro}
		\label{lem:m-iso-characterization}
		If $m\in\nbb$, $(X,\ascr,\mu)$ is a discrete measure space and $\cf\in\B[L^2(\mu)]$, then the following are equivalent{\em :}
		\begin{enumerate}
			\item $\cf$ is an $m$--isometry,
			\item for all $x\in X$,
			\begin{align}
			\label{wz1}
			\sum_{k=0}^m (-1)^k\binom mk \hn{k}(x)=0,
			\end{align}
			\item $\sum_{k=0}^m (-1)^k\binom mk \hn{n+k}(x)=0$ for all $x\in X$ and $n\in\zbb_+$,
			\item $\hn{n}(x)$ is a polynomial in $n$ of degree at most $m-1$ for all $x\in X$.
		\end{enumerate}
	\end{pro}
	
	Observe that, if $\operatorname{card}(\phi^{-1}(x)) = 0$ for some $x\in X$, then $\cf$ is not an $m$--isometry, as condition (ii) from Proposition~\ref{lem:m-iso-characterization} is not satisfied.
	
	Let us gather the following assumptions.
	\begin{gather}
	\tag{AS}
	\label{def:composition-operator}
	\begin{gathered}
	\begin{minipage} {75ex}
	Suppose that $\kappa\in\nbb$, $\eta \in \nbb \cup \{\infty\}$ and
	$$X = \{x_1, \ldots, x_{\kappa}\} \cup \bigcup_{i=1}^{\eta}\Big\{x_{i,j}\colon j \in \nbb\Big\},$$
	where $\{x_i\}_{i=1}^{\kappa}$ and $\{x_{i,j}\}_{i=1}^{\eta}{_{j=1}^{\infty}}$ are two disjoint systems of distinct points of $X$.
	Assume that $(X, \ascr, \mu)$ is a discrete measure space and a self-map $\phi$ of $X$ is defined by
	\end{minipage}   \\ \notag
	\begin{aligned}
	\phi(x) = \begin{cases}
	x_{i,j-1} & \text{ if } x = x_{i,j} \text{ for some } i\in J_{[1,\eta]} \text{ and } \, j\in \nbb\setminus \{1\},\\
	x_{\kappa} & \text{ if } x = x_{i,1} \text{ for some } i\in J_{[1,\eta]} \text{ or } x=x_1, \\
	x_{i-1} & \text{ if } x = x_{i} \text{ for some } i\in \{j\in\nbb\colon 2\le j \le  \kappa\}.
	\end{cases}
	\end{aligned}
	\end{gathered}
	\end{gather}
	An operator $\cf$ defined on a discrete measure space $(X,\ascr,\mu)$ satisfying \eqref{def:composition-operator} is called a \textit{composition operator on a directed graph with one circuit}.
	\begin{figure}
		\begin{tikzpicture}[-latex,auto,node distance =1.8 cm,
		on grid, state/.style={circle, draw,
			minimum width=0.7cm, fill=gray!30}]
		\node[state](BRANCH) {$x_\kappa$};
		\node[state](xk) [above left =of
		BRANCH]{$x_{\kappa-1}$}; \node[state](x1)
		[below left =of BRANCH]{$x_1$};
		
		\node[state](x11)[above right =of
		BRANCH]{$x_{1,1}$};
		\node[state](x12)[right =of
		x11]{$x_{1,2}$}; \node[state](x13)[right
		=of x12]{$x_{1,3}$}; \node(x14)[right
		=of x13]{$ $};
		
		\node[state](x21)[right =of
		BRANCH]{$x_{2,1}$};
		\node[state](x22)[right =of
		x21]{$x_{2,2}$}; \node[state](x23)[right
		=of x22]{$x_{2,3}$}; \node(x24)[right
		=of x23]{$ $};
		
		\node(end)[below right =of BRANCH]{$.$};
		\node(end1)[right =of end]{$.$};
		\node(end2)[right =of end1]{$.$};
		\node(end3)[right =of end2]{$ $};
		
		\path (x1) edge [bend left, dashed]
		node[below] {$ $} (xk); \path (BRANCH)
		edge [bend left] node[below] {$ $} (x1);
		\path (xk) edge [bend left] node[below]
		{$ $} (BRANCH);
		
		\path (BRANCH) edge (x11); \path (x11)
		edge (x12); \path (x12) edge (x13);
		\path[-] (x13) edge [dashed] (x14);
		
		\path (BRANCH) edge (x21); \path (x21)
		edge (x22); \path (x22) edge (x23);
		\path[-] (x23) edge [dashed] (x24);
		
		\path[-] (BRANCH) edge [dashed] (end);
		\path[-] (end1) edge [dashed] (end);
		\path[-] (end2) edge [dashed] (end1);
		\path[-] (end3) edge [dashed] (end2);
		\end{tikzpicture}
		\caption{\label{fig:composition-operator}
			The graph connected with a
			composition operator on a directed graph with one
			circuit.}
	\end{figure}
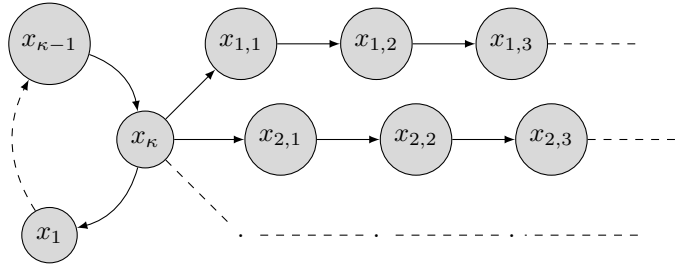
	
	For $\kappa\in\nbb$, denote by $\Phi_1\colon\zbb\to\zbb$ and $\Phi_2\colon\zbb\to\{1,\ldots,\kappa\}$ the functions uniquely determined by the following formula
	\begin{align*}
	n=\Phi_1(n)\kappa + \Phi_2(n),\quad n\in\zbb.
	\end{align*}
	It follows from the definition of functions $\Phi_1$ and $\Phi_2$ that
	\begin{align}
	\begin{aligned}
	\label{aa1}
	\Phi_1(l\kappa+1)&=\Phi_1(l\kappa+r), \quad l\in \zbb,\;r\in J_{[1,\kappa]},\\
	\Phi_2(l\kappa+r_1+r_2)&=\Phi_2(l\kappa +r_1)+r_2,\quad l\in \zbb,\;r_1\in
	\nbb,\;r_2\in \zbb_+,\; r_1+r_2\le \kappa.
	\end{aligned}
	\end{align}
	It is easily seen that the Radon-Nikodym
	derivative $\hn{}$ and $\hn{n}(x)$ for
	$x\in\{x_\kappa\}\cup\{x_{i,j},\colon i\in
	J_{[1,\eta]},\;j\in\nbb\}$ can be calculated as
	follows (cf.\ \cite[(3.4.6) \&
	(3.4.7)]{B-J-J-S-2017}):
	\begin{align}
	\label{rndx151}
	\hn{}(x) &=\begin{cases}
	\frac{\mu(x_1)+\sum_{i=1}^\eta\mu(x_{i,1})}{\mu(x_\kappa)} & \text{if } x=x_\kappa,\\[1ex]
	\frac{\mu(x_{r+1})}{\mu(x_r)} & \text{if } x=x_r \text{ for some } r\in J_{[1,\kappa-1]},\\[1ex]
	\frac{\mu(x_{i,j+1})}{\mu(x_{i,j})} &
	\text{if } x=x_{i,j} \text{ for some } i\in
	J_{[1,\eta]},\; j\in\nbb,
	\end{cases}
	\end{align}
	\begin{align}
	\label{rndx150}
	\hn{n}(x_\kappa) &= \frac{\mu(x_{\Phi_2(n)})}{\mu(x_\kappa)} + \sum_{i=1}^{\eta}\sum_{l=0}^{\Phi_1(n)} \frac{\mu(x_{i,l\kappa+\Phi_2(n)})}{\mu(x_k)}, \quad n\in\zbb_+, \\
	\label{rndx11} \hn{n}(x_{i,j})&=
	\frac{\mu(x_{x_{i,n+j}})}{\mu(x_{i,j})},
	\quad j\in\nbb,\;i\in J_{[1,\eta]},\;
	n\in\zbb_+.
	\end{align}
	Observe also that (cf.\ \cite[(3.4.7)]{B-J-J-S-2017})
	\begin{align}
	\label{hfnx0}
	\begin{aligned}
	\hn{n}(x_r)&=\frac{\mu(x_{1})}{\mu(x_r)}
	\hn{n+r-1}(x_{1}), \quad r \in J_{[1,\kappa]},\;
	n\in\zbb_+.
	\end{aligned}
	\end{align}
	Using \eqref{rndx150} and \eqref{hfnx0} we can prove that
	\begin{align}
	\label{rndx0}
	\hn{n}(x_{1}) &= \frac{\mu(x_{\Phi_2(n+1)})} {\mu(x_{1})} + \sum_{i=1}^{\eta}\sum_{l=0}^{\Phi_1(n+1)-1}
	\frac{\mu(x_{i,l\kappa+\Phi_2(n+1)})}{\mu(x_{1})}, \quad n\in\zbb_+.
	\end{align}
	
	Now we collect some necessary facts needed in the proof of Theorem~\ref{th:m-iso-branches}.
	First recall that (cf., \cite[p.\ 524]{jablonski-2003})
	\begin{align}
	\label{rndx2}
	\text{if $m\in\nbb$ and $j\in J_{[1,m]}$, then }\sum_{p=j}^m(-1)^p\binom mp =(-1)^j\binom{m-1}{j-1}.
	\end{align}
	The following result is used in the proof of Lemma~\ref{lem:technical-sum-decomposition}.
	
	\begin{lem}
		If $\kappa$, $p\in\nbb$ and $\{a_n\}_{n=1}^p\subseteq \rbb$, then
		\begin{align}
		\label{rndx1}
		\sum_{r=1}^{\kappa} \sum_{l=0}^{\Phi_1(p+r)-1} {a_{l\kappa+\Phi_2(p+r)}} = \sum_{j=1}^p a_j.
		\end{align}
	\end{lem}
	\begin{proof}
		Set
		\begin{align*}
		A&=\{r\in\zbb_+\colon 1\le r \le\kappa-\Phi_2(p)\},\\
		B&=\{r\in\zbb_+\colon \kappa-\Phi_2(p)+1\le r \le \kappa\},
		\end{align*}
		and observe that (use equality $p=\Phi_1(p)\kappa+\Phi_2(p)$ and \eqref{aa1})
		\begin{align}
		\label{phii1}
		\begin{aligned}
		&\Phi_1(p+r)=\Phi_1(p),
		&&\Phi_2(p+r)=\Phi_2(p)+r,  &&r\in A, \\
		&\Phi_1(p+r)=\Phi_1(p)+1,
		&&\Phi_2(p+r)=\Phi_2(p)+r-\kappa,
		&&r\in B,
		\end{aligned}
		\end{align}
		This implies that
		\begin{align} \notag
		\sum_{r\in B} \sum_{l=0}^{\Phi_1(p)} {a_{l\kappa+\Phi_2(p+r)}}
		&= \sum_{r\in B} \sum_{l=0}^{\Phi_1(p)} {a_{(l-1)\kappa+\Phi_2(p)+r}} \\
		\label{phii2}
		&= \sum_{r\in B} \sum_{l=0}^{\Phi_1(p)-1} {a_{l\kappa+\Phi_2(p)+r}}+\sum_{j=1}^{\Phi_2(p)}a_j,
		\end{align}
		and consequently
		\begin{align*}
		\sum_{r=1}^{\kappa}& \sum_{l=0}^{\Phi_1(p+r)-1} {a_{l\kappa+\Phi_2(p+r)}}\\
		&\overset{\eqref{phii1}}= \sum_{r\in A} \sum_{l=0}^{\Phi_1(p)-1} {a_{l\kappa+\Phi_2(p)+r}}+ \sum_{r\in B} \sum_{l=0}^{\Phi_1(p)} {a_{l\kappa+\Phi_2(p+r)}} \\
		& \overset{\eqref{phii2}}= \sum_{r=1}^{\kappa} \sum_{l=0}^{\Phi_1(p)-1} {a_{l\kappa+\Phi_2(p)+r}} +\sum_{j=1}^{\Phi_2(p)}a_j \\
		&\hspace{1.7ex}=\sum_{j=1}^{p}a_j,
		\end{align*}
		which completes the proof.
	\end{proof}
	
	Now we prove a technical lemma which is vital for later results.
	
	\begin{lem}
		\label{lem:technical-sum-decomposition}
		Suppose that $m\in\nbb$,
		\eqref{def:composition-operator} holds and
		$\sum_{i=1}^\eta\sum_{j=1}^m\mu(x_{i,j})<\infty$.
		Then
		\begin{align*}
		\sum_{r=1}^{\kappa} \mu(x_r)\sum_{p=0}^{m} (-1)^p\binom{m}{p} \hn{p}(x_r)
		&=-\sum_{i=1}^{\eta} \sum_{p=0}^{m-1}(-1)^p \binom{m-1}{p} \mu(x_{i,p+1}) \\
		&=-\sum_{i=1}^{\eta} \mu(x_{i,1}) \sum_{p=0}^{m-1} (-1)^p \binom{m-1}{p} \hn{p}(x_{i,1}).
		\end{align*}
	\end{lem}
	\begin{proof}
		First observe that by our assumptions
		$\hn{p}(x_r)\in\rbb_+$ for all $p\in J_{[1,m]}$
		and $r\in J_{[1,\kappa]}$, and for $r\in
		J_{[1,\kappa]}$, \allowdisplaybreaks
		\begin{align}
		\notag
		\mu(x_r)&\sum_{p=0}^{m} (-1)^p \binom{m}{p} \hn{p}(x_r) \\
		\notag &\overset{\eqref{hfnx0}} = \sum_{p=0}^{m} (-1)^p \mu(x_{1})\binom{m}{p}\hn{p+r-1} (x_{1}) \\
		\label{wz3}
		&\overset{\eqref{rndx0}}=\sum_{p=0}^{m} (-1)^p \binom{m}{p} \bigg({\mu(x_{\Phi_2(p+r)})} +
		\sum_{i=1}^{\eta}\sum_{l=0}^{\Phi_1(p+r)-1} {\mu(x_{i,l\kappa+\Phi_2(p+r)})} \bigg).
		\end{align}
		The latter, combined with the fact that
		\begin{align*}
		\sum_{r=1}^{\kappa}{\mu(x_{\Phi_2(p+r)})}=\sum_{r=1}^{\kappa} \mu({x_r}), \quad p\in J_{[0,m]},
		\end{align*}
		yields
		\begin{align*}
		\sum_{r=1}^{\kappa} &\mu(x_r)\sum_{p=0}^{m} (-1)^p \binom{m}{p} \hn{p}(x_r) \\
		& =\sum_{i=1}^{\eta} \sum_{p=0}^{m} (-1)^p \binom{m}{p} \sum_{r=1}^{\kappa} \sum_{l=0}^{\Phi_1(p+r)-1}{\mu(x_{i,l\kappa+\Phi_2(p+r)})}    \\
		& \overset{\eqref{rndx1}}=\sum_{i=1}^{\eta} \sum_{p=1}^{m} (-1)^p\binom{m}{p} \sum_{j=1}^{p} \mu(x_{i,j})\\
		& =\sum_{i=1}^{\eta} \sum_{j=1}^m\sum_{p=j}^{m} (-1)^p \binom{m}{p}\mu(x_{i,j})\\
		& \overset{\eqref{rndx2}}=\sum_{i=1}^{\eta}  \sum_{j=1}^{m} (-1)^j \binom{m-1}{j-1}\mu(x_{i,j})\\
		& =-\sum_{i=1}^{\eta} \sum_{j=0}^{m-1} (-1)^j \binom{m-1}{j}\mu(x_{i,j+1})\\
		& \overset{\eqref{rndx11}}=-\sum_{i=1}^{\eta} \mu(x_{i,1}) \sum_{j=0}^{m-1} (-1)^j \binom{m-1}{j}\hn{j}(x_{i,1}).
		\end{align*}
		This completes the proof.
	\end{proof}
	
	The below proposition is the main tool used in the proof of Theorem~\ref{th:m-iso-branches}.
	
	\begin{pro}
		\label{thm:poly-degree-branches} Suppose $m\ge2$, 
		\eqref{def:composition-operator} holds,
		$\mu(x_{i,j+1})$ is a polynomial in $j$ of
		degree at most $m-1$ for all $i\in
		J_{[1,\eta]}$ and $\sum_{i=1}^\eta\sum_{j=1}^m\mu(x_{i,j})<\infty$.
		Then $\mu(x_{i,j+1})$ is a
		polynomial in $j$ of degree at most $m-2$
		for all $i\in J_{[1,\eta]}$ if and only~if
		\begin{align} \label{kvnui2bdv}
		\sum_{r=1}^\kappa\mu(x_r)\sum_{p=0}^{m}
		(-1)^p \binom{m}{p} \hn{p}(x_r)=0.
		\end{align}
	\end{pro}
	\begin{proof}
		Set
		$\gammab_i=\{\mu(x_{i,j+1})\}_{j=0}^\infty$
		for $i\in J_{[1,\eta]}$. Applying
		Lemma~\ref{lem:polynomial} we deduce that
		$\triangle^m \gammab_i=0$ for $i\in
		J_{[1,\eta]}$. Hence, by
		Corollary~\ref{cor:delta-polynomial-constant},
		there exists a system $\{a_i\colon i\in
		J_{[1,\eta]}\}\subseteq\rbb_+$ such that
		\begin{align}
		\label{eq:delta-sequence-constant}
		(\triangle^{m-1} \gammab_i)_j=a_i,
		\quad j\in\nbb,\;i\in J_{[1,\eta]}.
		\end{align}
		Using Lemma~\ref{lem:technical-sum-decomposition} we get
		\allowdisplaybreaks
		\begin{align*}
		\sum_{r=1}^{\kappa} &\mu(x_r)\sum_{p=0}^{m} (-1)^p \binom{m}{p} \hn{p}(x_r)
		=-\sum_{i=1}^{\eta} \sum_{j=0}^{m-1} (-1)^j \binom{m-1}{j}\mu(x_{i,j+1})\\
		&\overset{\eqref{Newt1}}=(-1)^m\sum_{i=1}^{\eta} (\triangle^{m-1}\gamma_i)_0
		=(-1)^m\sum_{i=1}^{\eta}a_i.
		\end{align*}
		This combined with the fact that $a_i\ge
		0$ for $i\in J_{[1,\eta]}$,
		\eqref{eq:delta-sequence-constant} and
		Lemma~\ref{lem:polynomial} proves the required equivalence and completes the proof.
	\end{proof}
	
	The next theorem states that if $\cf$ is a bounded $m$--isometric composition operator on a directed graph with one circuit, then measure on branches form, in fact, polynomials of degree at most $m-2$.
	
	\begin{thm}
		\label{th:m-iso-branches} Suppose that
		\eqref{def:composition-operator} holds, $m$
		is an integer such that $m\ge2$ and
		$\cf\in\B[L^2(\mu)]$ is an $m$--isometry.
		Then $\mu(x_{i,j+1})$ is a polynomial in
		$j$ of degree at most $m-2$ for all $i\in
		J_{[1,\eta]}$.
	\end{thm}
	\begin{proof}
		Set
		$\gammab_i=\{\mu(x_{i,j+1})\}_{j=0}^\infty$
		for $i\in J_{[1,\eta]}$. Applying
		Proposition~\ref{lem:m-iso-characterization},
		substituting $x=x_{i,j}$, $i\in
		J_{[1,\eta]}$, $j\in\nbb$, into
		\eqref{wz1} and using \eqref{rndx11} and
		\eqref{Newt1} we deduce that
		$\triangle^m \gammab_i=0$ for all $i\in
		J_{[1,\eta]}$.
		Thus, by Lemma~\ref{lem:polynomial}, $\mu(x_{i,j+1})$ is a polynomial in $j$ of
		degree at most $m-1$ for all $i\in
		J_{[1,\eta]}$.
		It follows from Proposition~\ref{lem:m-iso-characterization} that \eqref{kvnui2bdv} holds, so by Proposition~\ref{thm:poly-degree-branches} the proof is completed.
	\end{proof}
	
	The following theorem, which is the main result of this section, establishes a characterization of $m$--isometric composition operators on directed graphs with one circuit.
	
	\begin{thm}
		\label{th:m-iso-characterization}
		Suppose \eqref{def:composition-operator} holds, $m$ is an integer such that $m\ge 2$ and $\cf \in \B[L^2(\mu)]$.
		Then the following conditions are equivalent{\em :}
		\begin{enumerate}
			\item $\cf$ is an $m$--isometry,
			\item  $\mu(x_{i,j+1})$ is a polynomial in $j$ of degree at most $m-2$ for all $i\in J_{[1,\eta]}$ and
			\begin{align}
			\label{wzo1} \sum_{p=0}^{m} (-1)^p
			\binom{m}{p}
			\bigg({\mu(x_{\Phi_2(p+r)})} +
			\sum_{i=1}^{\eta}\sum_{l=0}^{\Phi_1(p+r)-1}
			{\mu(x_{i,l\kappa+\Phi_2(p+r)})}
			\bigg)=0,\; r\in J_{[1,\kappa]}.
			\end{align}
		\end{enumerate}
	\end{thm}
	\begin{proof}
		(i)$\Rightarrow$(ii).
		Apply Theorem~\ref{th:m-iso-branches}, Proposition~\ref{lem:m-iso-characterization} and \eqref{wz3}.
		
		(ii)$\Rightarrow$(i).
		This can by proved by using \eqref{hfnx0}, \eqref{wz3}, \eqref{wzo1}, Proposition~\ref{lem:m-iso-characterization}, Lemma~\ref{lem:polynomial} and the following observation: if $a_j$ is a polynomial in $j$ of degree at most $k$, then $a_{j+n}$ is a polynomial in $j$ of degree at most $k$ for all $n\in\zbb_+$.
	\end{proof}
	
	Our next result is a consequence of Proposition \ref{thm:poly-degree-branches} and Theorem \ref{th:m-iso-characterization}.
	
	\begin{cor}
		\label{cor:m-iso-characterization-one-element}
		Suppose that
		\eqref{def:composition-operator} holds with
		$\kappa=1$, $\cf\in\B[L^2(\mu)]$ and $m$ is
		an integer such that $m\ge2$. Then $\cf$ is
		an $m$--isometry if and only if
		$\mu(x_{i,j+1})$ is a polynomial in $j$ of
		degree at most $m-2$ for all $i\in
		J_{[1,\eta]}$. Moreover, if there exists
		$i_0\in J_{[1,\eta]}$ such that
		$\mu(x_{i_0,j+1})$ is a polynomial in $j$
		of degree $m-2$, then $\cf$ is a strict
		$m$--isometry.
	\end{cor}
	
	Note that, under the assumptions of Corollary~\ref{cor:m-iso-characterization-one-element}, if $\cf\in\B[L^2(\mu)]$ is $m$--isometric, then for every $t\in\rbb_+$, $\cf\in\B[L^2(\mu_t)]$ is also $m$--isometric, where $\mu_t$ is the measure on $\ascr$ uniquely determined by
	\begin{align*}
	\mu_t(x)=
	\begin{cases}
	\mu(x)+t & \text{if } x=x_1, \\
	\mu(x) & \text{otherwise.}
	\end{cases}
	\end{align*}
	
	Below we show that bounded composition operators on directed graphs with one circuit separate the class of $m$--isometric operators for $m\ge 2$.
	
	\begin{exa}
		Suppose \eqref{def:composition-operator} holds with $\kappa=1$, $\eta\in J_{[1,\infty]}$ and $m\ge2$ is an integer.
		Let $q$ be a real polynomial of degree $m-2$ such that all coefficients of $q$ are nonnegative, and set $\mu(x_1)=1$ and
		\begin{align*}
		\mu(x_{i,j}) = 2^{-i} q(j), \quad i\in
		J_{[1,\eta]},\; j\in\nbb.
		\end{align*}
		Observe that by \eqref{rndx151},
		\begin{align*}
		\hn{}(x) =
		\begin{cases}
		1 + q(1)\sum_{i\in J_{[1,\eta]}} 2^{-i} & \text{if } x=x_{1}, \\[1ex]
		\frac{q(j+1)}{q(j)} & \text{if }
		x=x_{i,j} \text{ for some } i\in J_{[1,\eta]}
		\text { and } j\in\nbb.
		\end{cases}
		\end{align*}
		Since $\lim_{j\to\infty}\frac{q(j+1)}{q(j)}=1$, this and \eqref{eq:boundedness-condition} implies that $\cf\in\B[L^2(\mu)]$.
		It follows from Corollary~\ref{cor:m-iso-characterization-one-element}, that $\cf$ is strictly $m$--isometric.
	\end{exa}
	
	The result below provides a sufficient condition for a bounded composition operator on a directed graph with one circuit to be an $m$--isometry.
	
	\begin{pro}
		\label{pro:m-exp-m-iso}
		Suppose \eqref{def:composition-operator} holds, $\cf\in\B[L^2(\mu)]$ and $m\ge2$ is an integer.
		If $(-1)^m\bscr_m(\cf)\le 0$ then $\bscr_{m}(\cf)= 0$.
	\end{pro}
	\begin{proof}
		Assume that $(-1)^m\bscr_m(\cf)\le 0$.
		This, \eqref{aa} and \cite[Lemma~2.1]{jablonski-2003} imply that
		\begin{align}
		\label{mexmis1}
		(-1)^m\sum_{k=0}^m (-1)^k \binom{m}{k} \hn{k}\le 0,
		\end{align}
		and hence, by \eqref{eq:newton-symbol-inequality},
		\begin{align}
		\label{mexmis2}
		(-1)^m\sum_{k=0}^{m-1} (-1)^k \binom{m-1}{k} \hn{k}\le 0.
		\end{align}
		Using Lemma~\ref{lem:technical-sum-decomposition} we deduce that
		\begin{align}
		\notag
		0\overset{\eqref{mexmis1}}\ge \sum_{r=1}^{\kappa} &\mu(x_r)(-1)^m\sum_{p=0}^{m} (-1)^p \binom{m}{p} \hn{p}(x_r) \\
		\notag
		&=-\sum_{i=1}^{\eta} \mu(x_{i,1}) (-1)^m\sum_{p=0}^{m-1} (-1)^p \binom{m-1}{p} \hn{p}(x_{i,1}) \overset{\eqref{mexmis2}}\ge0,
		\end{align}
		and thus, it follows from \eqref{mexmis1} and
		\eqref{mexmis2} that
		\begin{align}
		\label{mexmis4}
		\sum_{p=0}^{m} (-1)^p \binom{m}{p} \hn{p}(x_r)=0, \quad r\ \in J_{[1,\kappa]}, \\
		\label{mexmis3} \sum_{p=0}^{m-1} (-1)^p
		\binom{m-1}{p} \hn{p}(x_{i,1}) =0, \quad
		i\in J_{[1,\eta]}.
		\end{align}
		Observe now that
		\begin{align*}
		0 &\overset{\eqref{mexmis1}} \ge (-1)^m \sum_{p=0}^{m} (-1)^p \binom{m}{p} \hn{p}(x_{i,j}) \\
		&= (-1)^m\sum_{p=0}^{m-1} (-1)^p \binom{m-1}{p} \hn{p}(x_{i,j}) \\
		&-(-1)^m
		\frac{\mu(x_{i,j+1})}{\mu(x_{i,j})}
		\sum_{p=0}^{m-1} (-1)^p \binom{m-1}{p}
		\hn{p}(x_{i,j+1}), \quad i\in J_{[1,\eta]},\;
		j\in\nbb.
		\end{align*}
		This and an induction argument on $j$ implies that
		\begin{align*}
		0 & \overset{\eqref{mexmis2}}\ge(-1)^m\sum_{p=0}^{m-1} (-1)^p \binom{m-1}{p} \hn{p}(x_{i,j+1}) \\
		&\ge (-1)^m
		\frac{\mu(x_{i,1})}{\mu(x_{i,j+1})}
		\sum_{p=0}^{m-1} (-1)^p \binom{m-1}{p}
		\hn{p}(x_{i,1}) \overset{\eqref{mexmis3}}=0,
		\quad i\in J_{[1,\eta]},\; j\in\nbb,
		\end{align*}
		and consequently
		\begin{align*}
		\sum_{p=0}^{m-1} (-1)^p \binom{m-1}{p}
		\hn{p}(x_{i,j+1})=0, \quad i\in J_{[1,\eta]},\;
		j\in\nbb.
		\end{align*}
		Combined with \eqref{mexmis4}, \eqref{mexmis3} and Proposition~\ref{lem:m-iso-characterization}, this completes the proof.
	\end{proof}
	
	Since bounded $2$--isometric operators are completely hyperexpansive and completely hyperexpansive are $2$--expansive, the following corollary is an immediate consequence of Proposition~\ref{pro:m-exp-m-iso}.
	
	\begin{cor}
		\label{cor:complete-hyperex}
		If \eqref{def:composition-operator} holds and $\cf\in\B[L^2(\mu)]$ then $\cf$ is completely hyperexpansive if and only if it is $2$--isometric.
	\end{cor}
	
	Concerning Proposition~\ref{pro:m-exp-m-iso}, one may ask if $m\ge2$ and $\cf\in\B[L^2(\mu)]$ are such that $(-1)^m\bscr_m(\cf)\ge 0$ implies $\bscr_{m}(\cf)= 0$.
	The following example shows that this is not true.
	
	\begin{exa}
		Assume \eqref{def:composition-operator} holds with $\kappa=1$ and $\eta=1$ and suppose that $m\ge2$ is an integer.
		Let $p$ be a polynomial of degree $m-1$ such that $p(j)>0$ for all $j\in\nbb$.
		Set $\mu(x_{1, j})=p(j)$ for $j\in\nbb$ and $\mu(x_1)=1$.
		Then, by Corollary~\ref{cor:m-iso-characterization-one-element} the operator $\cf$ is strictly $(m+1)$--isometric, in particular, it is not $m$--isometric.
		It follows from Proposition~\ref{th:alternating-expansiveness} that $(-1)^m\bscr_m(\cf)\ge 0$.
	\end{exa}

	\section{$m$--isometric completion problem}

	In this section we discuss some results related to $m$--isometric completion problem.
	The counterparts of it for unilateral weighted shifts appeared in e.g. \cite{A-L-2016,J-J-K-S-2011}.
	In \cite[Sections 4 and 5]{J-J-K-S-2011} authors studied completely hyperexpansive and 2--isometric completion problems for unilateral weighted shifts.
	Abdullah and Le proved that for each sequence $\{a_n\}_{n=1}^m\subset\C\setminus\{0\}$ there exists an $(m+2)$--isometric unilateral weighted shift with weight sequence starting with $\{a_n\}_{n=1}^m$ (see \cite[Proposition~2.7]{A-L-2016}).
	In this section we focus also on the $m$--isometric completion problem for composition operators on directed graphs with one circuit.
	
	For further references we need a technical lemma.
	\begin{lem}
		\label{lem:completion-interpolation} Let
		$l\in\zbb_+$ and $\{b_n\}_{n=0}^{l}
		\subseteq(0,\infty)$. Then there exists
		$c\in(0,\infty)$ such that for every
		$t\in[c,\infty)$ there exists a polynomial
		$w_t$ in one indeterminate $x$ with real
		coefficients of degree $l+1$ such that
		$w_t(n) = b_n$ for $n\in J_{[0,l]}$,
		$w_t(l+1)=t$ and $w_t(n) > 0$ for $n \in
		J_{[l+2,\infty]}$.
	\end{lem}
	\begin{proof}
		We use induction on $l$. The
		case $l=0$ is obvious. Assume that lemma
		holds for a fixed unspecified $l\in\nbb$
		and let
		$\{b_n\}_{n=0}^{l+1}\subseteq(0,\infty)$.
		By the induction hypothesis there exists a
		polynomial $v$ in one indeterminate $x$
		with real coefficients of degree $l+1$ such
		that $v(n) = b_n$ for $n\in J_{[0,l]}$,
		$v(l+1)\neq b_{l+1}$ and $v(n) > 0$ for $n
		\in J_{[l+1,\infty]}$. For
		$s\in[1,\infty)$, let
		\begin{align*}
		w_s(x)=v(x)+\Big(s|\alpha|(x-l-1)+{\alpha}\Big)
		\frac{\prod_{j=0}^{l}(x-j)}{\prod_{j=0}^{l}(l+1-j)},
		\quad x\in \rbb,
		\end{align*}
		where $\alpha = b_{l+1}-v(l+1)$. It is
		easily seen that $w_s$ is a polynomial in
		one indeterminate $x$ with real
		coefficients of degree $l+2$ and for any
		$s\ge1$, $w_s(n)=b_n$ for $n\in
		J_{[0,l+1]}$ and $w_s(n) > 0$ for $n\in
		J_{[l+2,\infty]}$. Since the function
		$\psi\colon [1,\infty)\ni s\to w_s(l+2)\in
		\rbb_+$ is continuous and
		$\lim_{s\to\infty}\psi(s)=+\infty$, the
		proof is completed.
	\end{proof}

	For $k\in J_{[1,\infty]}$ define a transformation
	$\beta\colon\rbb^{J_{[1,k]}}\to \rbb^{J_{[0,k]}}$
	by
	\begin{align*} 
	(\beta \gammab)_n=\prod_{i=1}^n\gamma^2_{i},\quad
	n\in J_{[0,k]},\,
	\gammab=\{\gamma_n\}_{n=1}^k\in\rbb^{J_{[1,k]}}.
	\end{align*}
	Recall that an operator $S\in\B[\ell_2]$ is called a \textit{unilateral weighted shift} with weights $\{\lambda_n\}_{n=1}^\infty\subseteq\C\setminus\{0\}$, if $Se_n = \lambda_{n+1}e_{n+1}$ for all $n\in\zbb_+$, where $\{\lambda_n\}_{n=1}^\infty$ is a bounded sequence.
	It is a matter of a straightforward verification to see that these operators are bounded and injective.
	
	Let $\mathcal{C}$ be a class of operators.
	For a classical weighted shift, the completion problem within the class $\mathcal{C}$ entails determining whether or not a given initial finite sequence of positive weights may be extended to the sequence of weights of an injective, bounded unilateral weighted shift which belongs to the class $\mathcal{C}$; such a shift is called a $\mathcal{C}$ class completion of the initial weight sequence.
	
	Now we provide the solution of the completion
	problem within the class of $m$--isometries for
	the classical weighted shift (cf.\
	\cite[Proposition~2.7]{A-L-2016}). Since a
	unilateral weighted shift with weights
	$\{\lambda_n\}_{n=1}^\infty\subseteq\C\setminus\{0\}$
	is unitarily equivalent to the shift with weight
	sequence
	$\{|\lambda_n|\}_{n=1}^\infty\subseteq(0,\infty)$,
	the following result can be generalized to
	unilateral weighted shifts with non-zero weights.

	\begin{pro}
		\label{pro:completion-shifts} Let
		$m,k\in\nbb$ and let
		$\tilde\lambdab:=\{\lambda_n\}_{n=1}^{k}\subseteq(0,
		\infty)$.
		Then
		\begin{enumerate}
			\item if $k\le m-2$, then there exists $c>0$ such
			that for every $r\in[c,\infty)$ the sequence
			$\{\lambda_n\}_{n=1}^{m-1}$ with an arbitrary
			sequence
			$\{\lambda_n\}_{n=k+1}^{m-2}\subseteq(0,\infty)$
			and $\lambda_{m-1}=r$ admits a strict
			$m$--isometric completion,
			\item if $k > m-2$, then $\tilde\lambdab$
			admits an $m$--isometric completion if and only
			if $\sum_{j=0}^{m} (-1)^j\binom {m}j (\beta
			\tilde\lambdab)_{n+j}=0$ for $n\in J_{[0,k-m]}$
			and the unique polynomial $w$ of degree at most
			$m-1$ which satisfies the conditions $w(n)=(\beta
			\tilde\lambdab)_{n}$ for $n\in J_{[0,m-1]}$, also
			satisfies $w(n)>0$ for all $n\in J_{[m,\infty]}$.
			Moreover, this completion is strict if and only
			if the degree of $w$ is equal to $m-1$.
		\end{enumerate}
	\end{pro}
	\begin{proof}
		(i). If $k < m-2$, let us fix an arbitrary
		sequence $\{\lambda_n\}_{n=k+1}^{m-2}
		\subseteq(0,\infty)$. Let
		$\hat\lambdab:=\{\lambda_n\}_{n=1}^{m-2}$. It
		follows from
		Lemma~\ref{lem:completion-interpolation} applied
		to the sequence $\{(\beta
		\hat\lambdab)_n\}_{n=0}^{m-2}$ that there exists
		$d\in(0,\infty)$ such that for every
		$t\in[d,\infty)$ there exists a polynomial $w_t$
		in one indeterminate $x$ with real coefficients
		of degree $m-1$ such that $w_t(n) = (\beta
		\hat\lambdab)_n$ for $n\in J_{[0,m-2]}$,
		$w_t(m-1)=t$ and $w_t(n)
		> 0$ for $n\in J_{[m,\infty]}$. Let
		$\lambda_n=\sqrt\frac{w_t(n)}{w_t(n-1)}$ for
		$n\in J_{[m-1,\infty]}$ and let $S$ be the
		unilateral weighted shift with weights
		$\lambdab:=\{\lambda_n\}_{n=1}^\infty$. Since
		$\{(\gammab_{S,e_0})_n\}_{n=0}^{\infty}
		=\{w_t(n)\}_{n=0}^\infty$, it follows from
		\cite[Theorem~2.1]{A-L-2016} that $S$ is a strict
		$m$--isometric completion of the weight sequence
		$\{\lambda_n\}_{n=1}^k$. Now it is a routine
		matter to check that (i) holds.
		
		(ii). In view of \cite[Theorem~2.1]{A-L-2016} and the
		uniqueness part of the Lagrange interpolation
		formula (see \cite[Section 2.5]{davis-1963}),
		$\tilde\lambdab$ admits an $m$--isometric
		completion if and only if there exists a
		polynomial $w$ in one indeterminate $x$ with real
		coefficients of degree at most $m-1$ such that
		$w(n)=(\beta\tilde\lambdab)_n$ for $n\in
		J_{[0,m-1]}$ and $w(n)>0$ for $n\in
		J_{[m,\infty]}$. Now, Lemma~\ref{lem:polynomial}
		completes the proof.
		
		The moreover part follows directly from \cite[Corollary~2.3]{A-L-2016}.
	\end{proof}
	
	We present an example which shows that for each the $m\ge2$ there exists $\{\lambda_n\}_{n=1}^{m}\subseteq(0,\infty)$ such that $(m+1)$--isometric completion problem does not have a solution (cf.\ \cite[Remark~2.8]{A-L-2016}).
	
	\begin{exa}
		First, assume that $m\ge2$ is an odd number.
		Let
		\begin{align*}
		w_m(x) = -\frac{(x-(m+1))^m}{(m+1)^m}.
		\end{align*}
		Since $w_m(i)>0$ for $i=0,\dots,m$, the sequence $\{\lambda_n\}_{n=1}^{m}$ given by the formula
		\begin{align*}
		\lambda_i = \sqrt{\frac{w_m(i)}{w_m(i-1)}}, \quad i\in J_{[1,m]},
		\end{align*}
		is well-defined and consists of positive real numbers.
		This implies that
		\begin{align*}
		w_m(n)=\prod_{i=1}^{n} \lambda_i^2, \quad n\in J_{[0,m]}.
		\end{align*}
		Combined with Proposition~\ref{pro:completion-shifts}(ii), this implies that there is no solution of the $(m+1)$--isometric completion problem for $\{\lambda_n\}_{n=1}^{m}$.
		
		If $m$ is even, then it is enough to consider
		\begin{align*}
		w_m(x) = -\frac{(x-(m+1))^{m-1}(x+(m-1))}{(m+1)^{m-1}(m-1)}.
		\end{align*}
		We leave the details to the reader.
	\end{exa}
	
	Now we concentrate on the completion problem for $m$--isometric composition operators on directed graphs with one circuit.
	
	\begin{pro}
		\label{pro:completion-composition} Suppose
		$m$, $\kappa\in\zbb$ are such that
		$\kappa>m\ge2$. Let $\eta$, $X$, $\ascr$
		and $\phi$ be as in
		\eqref{def:composition-operator},
		$M\in(0,\infty)$ and $\{w_i\}_{i=1}^\eta$ be a system of polynomials
		of degree at most $m-2$ such that
		$w_i(j)>0$ for all $i\in J_{[1,\eta]}$ and
		$j\in\nbb$ and
		\begin{align}
		\label{eq:completion-boundedness}
		\max\bigg\{\sum_{i=1}^\eta w_i(1),
		\sup\Big\{\frac{w_i(j+1)}{w_i(j)}\colon
		i\in J_{[1,\eta]},\; j\in
		\nbb\Big\}\bigg\}\le M.
		\end{align}
		Then there exists a measure $\mu$ on $\ascr$
		such that $\mu(x_{i,j})=w_i(j)$ for $i\in
		J_{[1,\eta]}$, $j\in\nbb$ and $\cf\in
		\B[L^2(\mu)]$ is an
		$m$--isometry. Moreover, there exists
		$t_0\in\rbb$ such that all measures $\mu_t$ having the above properties
		can be parameterized by
		$t\in(t_0,\infty)$, where
		\begin{align*}
		\mu_t(x):=\begin{cases}
		\mu(x)+t & \text{if } x=x_j \text{ with } j\in J_{[1,\kappa]}, \\
		\mu(x) & \text{otherwise.} \\
		\end{cases}
		\end{align*}
	\end{pro}
	\begin{proof}
		In view of \eqref{eq:boundedness-condition}, \eqref{rndx151} and Theorem~\ref{th:m-iso-characterization} any solution of \eqref{wzo1} fulfills our requirements.
		Observe that \eqref{wzo1} can be written in the following matrix form
		\begin{align}
		\label{eq:completion-matrix-equation}
		\begin{bmatrix}
		a_0 & a_1 & \cdots & a_{\kappa-2} & a_{\kappa-1}   \\
		a_{\kappa-1} & a_0 & a_1   &  & a_{\kappa-2}   \\
		\vdots  & \ddots & \ddots & \ddots &\vdots  \\
		a_2  &  & \ddots & a_0 & a_1 \\
		a_1 & a_2 & \cdots & a_{\kappa-1} & a_0  \\
		\end{bmatrix}
		\begin{bmatrix}
		\mu(x_1) \\
		\mu(x_{2}) \\
		\vdots \\
		\vdots \\
		\mu(x_{\kappa}) \\
		\end{bmatrix} =
		\begin{bmatrix}
		b_1 \\
		b_2 \\
		\vdots \\
		\vdots \\
		b_k \\
		\end{bmatrix},
		\end{align}
		where $a_p=(-1)^{p}\binom{m}{p}$ for $p\in J_{[0,m]}$, $a_p=0$ for $p>m$ and
		\begin{align*}
		b_r=\sum_{p=0}^m(-1)^{p+1}\binom{m}{p}\sum_{i=1}^{\eta}\sum_{l=0}^{\Phi_1(p+r)-1}
		{\mu(x_{i,l\kappa+\Phi_2(p+r)})}, \quad
		r\in J_{[1,\kappa]}.
		\end{align*}
		Since the matrix in \eqref{eq:completion-matrix-equation}, call it $A$, is circulant, it follows from \cite[Proposition~1.1]{ingleton-1956} that the rank of $A$ is equal to $\kappa-1$ because the associated polynomial of $A$ is  $(x-1)^m$ and the degree of the greatest common divisor of $1-x^\kappa$ and $(x-1)^m$ is equal to $1$.
		By Proposition~\ref{thm:poly-degree-branches}, the rank of the augmented matrix of \eqref{eq:completion-matrix-equation} is equal to $\kappa-1$, hence the solutions of the system \eqref{eq:completion-matrix-equation} form one dimensional affine subspace of $\rbb^\kappa$.
		Note that the vector $v=[t,\ldots,t]$ for $t\in\rbb$ satisfies equality $Av=0$. The proof is completed.
	\end{proof}
	
	By using similar reasoning we prove the following useful characterization of $2$--isometric and $3$--isometric composition operators on directed graphs with one circuit, which provides an explicit description of measure $\mu$ on $X$.
	
	\begin{thm}
		\label{th:characterization_3_isom}
		Suppose \eqref{def:composition-operator} holds with $\kappa>1$.
		Then the following are equivalent{\em :}
		\begin{enumerate}
			\item $\cf \in \B[L^2(\mu)]$ is $3$--isometric,
			\item there exist $M$, $t\in(0,\infty)$, a system of polynomials $\{w_i\}_{i=1}^\eta$ of degree at most $1$ and two systems $\{c_i\}_{i=1}^\eta\subseteq(0,\infty)$ and
			$\{d_i\}_{i=1}^\eta\subseteq\rbb_+$ such that
			\eqref{eq:completion-boundedness} holds,
			$c:=\sum_{i=1}^\eta c_i<\infty$,
			$d:=\sum_{i=1}^\eta d_i<\infty$ and
			\begin{align} \notag
			\mu(x_{i,j})&=c_i+d_i(j-1), \quad i\in J_{[1,\eta]},\;j\in\nbb, \\
			\label{defmu2}
			\mu(x_i)&=w_{c,d}^{\kappa,t}(i),
			\quad i\in J_{[1,\kappa]},
			\end{align}
			where $w_{c,d}^{\kappa,t}$ is a polynomial of
			degree at most $2$ given by
			\begin{align}
			\label{eq:3-iso-polynomial}
			w_{c,d}^{\kappa,t}(x)=\frac{d}{2\kappa}x^2 +\Big(\frac{c}{\kappa}-\frac{\kappa+2}{2\kappa}d\Big)x - c + d +t, \quad x\in\rbb.
			\end{align}
		\end{enumerate}
		Moreover, $\cf$ is $2$--isometric if and
		only if {\em (ii)} holds with a system of
		polynomials $\{w_i\}_{i=1}^\eta$ of
		degree equal to $0$ and with $d_i=0$ for all
		$i\in J_{[1,\eta]}$.
	\end{thm}
	\begin{proof}
		Let us first assume that $\kappa\ge 4$.
		In view of Proposition~\ref{pro:completion-composition} and its proof it is enough to show that a vector $[\mu(x_i)]_{i=1}^\kappa$ given by \eqref{defmu2} is a solution of \eqref{eq:completion-matrix-equation}.
		This can be proved by using the fact that $\Delta^3 w_{c,d}^{\kappa,t}=0$.
		It is a matter of direct verification that our result is true for $\kappa=2,3$.
	\end{proof}
	
	Using the above tools we are ready to prove some results regarding the $m$--isometric completion problem for composition operators on directed graphs with one circuit.
	One can think of many ways to state the $m$--isometric completion problem in the case of these composition operators, thus we limit our considerations to only selected number of possibilities.
	
	First we begin with the simplest situation, namely, when the circuit has only one element.
	In this situation we can prove a result, similar in its nature, to analogical completion problem for unilateral weighted shifts.
	
	\begin{pro}
		Let $m\in\nbb$ and assume that $\{a_n\}_{n=1}^{m}\subseteq (0,\infty)$.
		Then there exist a discrete measure space $(X,\ascr,\mu)$ and a self-map $\phi\colon X\to X$ satisfying \eqref{def:composition-operator} with $\kappa=1$ and $\eta=1$ such that $\mu(x_{1,n})=a_n$ for $n\in J_{[1,m]}$ and
		$\cf$ is an $(m+2)$--isometry.
		Moreover, if $m=1$, then $(X,\ascr,\mu)$ and $\phi\colon X\to X$ can be chosen so that $\cf$ is a $2$--isometry.
	\end{pro}
	\begin{proof}
		If $m\ge 2$, then we use Lemma~\ref{lem:completion-interpolation} for $\{a_n\}_{n=1}^{m}$ and Corollary~\ref{cor:m-iso-characterization-one-element}.
		If $m=1$, then the result follows directly from Corollary~\ref{cor:m-iso-characterization-one-element}.
	\end{proof}
	
	Let us now prove another result regarding $2$--isometric completion problem in which we assume that measure of elements located in the circuit is a priori given.
	
	\begin{pro}
		\label{pro:completion-circuit}
		Suppose that $\kappa>1$, $\eta\in\nbb\cup\{\infty\}$ and let  $a:=\{a_n\}_{n=1}^\kappa\subseteq(0,\infty)$.
		Then the following are equivalent:
		\begin{enumerate}
			\item there exist a discrete measure space $(X,\ascr,\mu)$ and a self-map $\phi\colon X\to X$ satisfying \eqref{def:composition-operator} such that $\mu(x_{n})=a_n$ for $n\in J_{[1,\kappa]}$ and $\cf\in\B[L^2(\mu)]$ is a $2$--isometry,
			\item $\{(\triangle a)_n\}_{n=1}^{\kappa-1}$ is a constant sequence such that $(\triangle a)_1>0$.
		\end{enumerate}
	\end{pro}
	\begin{proof}
		(i)$\Rightarrow$(ii).
		This implication is a straightforward application of Theorem~\ref{th:characterization_3_isom}.
		Indeed, if $\cf$ is a $2$--isometry, then equation \eqref{eq:3-iso-polynomial} implies that $\mu(x_i)=\frac{c}{\kappa}i-c +t$ for $i\in J_{[1,\kappa]}$, where $c=\sum_{i=1}^\eta \mu(x_{i,1})$.
		Thus $\{(\triangle a)_n\}_{n=1}^{\kappa-1}$ is a constant sequence and $(\triangle a)_1=\mu(x_2) - \mu(x_1) =\frac{c}{\kappa}>0$. 
		
		(ii)$\Rightarrow$(i).
		Define $X$ and $\ascr$ as in \eqref{def:composition-operator}.
		Let $\mu$ be a discrete measure on $\ascr$ such that
		\begin{align}
		\label{eq:completion-measure}
		\mu(x) = \begin{cases}
		a_n & \text{if } x=x_n \text{ for some } n\in J_{[1,\kappa]}, \\
		r_i(\triangle a)_1 & \text{if }
		x=x_{i,j} \text{ for some } i\in J_{[1,\eta]}
		\text{ and } j\in\nbb,
		\end{cases}
		\end{align}
		where
		\begin{align}
		\label{eq:completion-measure-2}
		r_i:= \begin{cases}
		2^{-i} \kappa & \text{if } \eta=\infty, \\
		\frac{\kappa}{\eta} & \text{if } \eta <\infty.
		\end{cases}
		\end{align}
		Now, set $c=\kappa(\triangle a)_1$ and note that \eqref{eq:completion-measure} combined with \eqref{eq:completion-measure-2} imply that $c=\sum_{i=1}^\eta\mu(x_{i,1})$.
		Define $t=a_1+c(1-\frac{1}{\kappa})$ and $w_{c,d}^{\kappa,t}$ as in \eqref{eq:3-iso-polynomial} with $d=0$.
		Then, since $\{(\triangle a)_n\}_{n=1}^{\kappa-1}$ is a constant sequence, 
		\begin{align*}
		\mu(x_{i}) & \overset{\eqref{eq:completion-measure}}= a_i = a_1+(\triangle a)_1(i-1) = a_1 + \frac{c}{\kappa} i - \frac{c}{\kappa} = t - c(1-\frac{1}{\kappa}) + \frac{c}{\kappa} i - \frac{c}{\kappa} \\
		&= \frac{c}{\kappa}i -c + t =
		w_{c,d}^{\kappa,t}(i), \quad i\in J_{[1,\kappa]}.
		\end{align*}
		Theorem~\ref{th:characterization_3_isom} implies that $\cf$ is a $2$--isometry.
	\end{proof}
	
	Let us note that, if $\eta > 1$ in Proposition~\ref{pro:completion-circuit}, then, if exists, the $2$--isometric operator that solves the completion problem is not unique.
	Indeed, let $(X,\ascr,\mu)$ be as in Proposition~\ref{pro:completion-circuit}(i).
	For $t\in(0,1)$ define a measure $\nu_t\colon\ascr\to(0,\infty)$ in the following way:
	\begin{align*}
	\nu_t(x) = \begin{cases}
	\mu(x_{1,j}) + t\mu(x_{2,j}) & \text{if } x=x_{1,j} \text{ for some } j\in \nbb, \\
	(1-t)\mu(x_{2,j}) & \text{if } x=x_{2,j} \text{ for some } j\in \nbb, \\
	\mu(x) & \text{otherwise}.
	\end{cases}
	\end{align*}
	It is an easy observation that $\cf$ satisfies Proposition~\ref{pro:completion-circuit}(i) with $\nu_t$ in place of $\mu$.

	\section{Subnormality of Cauchy dual of $2$--isometry}

	We begin this section with characterization of analytic composition operators on directed graphs with one circuit.
	The equivalence (i)$\Leftrightarrow$(ii) of the below lemma is a counterpart of a known result that describes $\RngInf(\cf):=\bigcap_{n=1}^\infty \cf^n(L^2(\mu))$ for a composition operator $\cf$ (cf.\ \cite[Remark 45]{B-J-J-S-2015} and \cite[Lemma~2.3]{A-C-T-2018}).
	
	\begin{lem}
		\label{lem:analyticity}
		Suppose \eqref{def:composition-operator} holds, $\cf\in\B[L^2(\mu)]$ and $f\in L^2(\mu)$.
		Then the following are equivalent{\em :}
		\begin{enumerate}
			\item $f\in\RngInf(\cf)$,
			\item \label{eq:analyticity_1}
			$f(x_{r}) = f(x_{i, l\kappa+r})$ for all
			$r\in J_{[1,\kappa]}$, $i\in J_{[1,\eta]}$ and
			$l\in\zbb_+$,
			\item $\cf^\kappa f=f$.
		\end{enumerate}
	\end{lem}
	\begin{proof}
		First observe that
		\begin{align}
		\label{eq:analyticity_3}
		\phi^{r+n}(x_{i,n})&=x_{\kappa-r},\quad n\in\nbb, \; r\in J_{[0,\kappa-1]},\; i\in J_{[1,\eta]}, \\
		\label{eq:analyticity_4}
		\phi^{l\kappa}(x_r) &= x_r, \quad
		l\in\zbb_+, r\in J_{[1,\kappa]}.
		\end{align}
		
		(i)$\Rightarrow$(ii).
		Assume that $f\in\RngInf(\cf)$ and fix $l\in\zbb_+$.
		Then there exists $g\in L^2(\mu)$ such that
		\begin{align}
		\label{eq:analyticity_2}
		f=\cf^{\kappa(l+1)}g.
		\end{align}
		Thus, for $r\in J_{[1,\kappa]}$ and $i\in
		J_{[1,\eta]}$ we get that
		\begin{align*}
		f(x_{i, l\kappa+r}) &\overset{\eqref{eq:analyticity_2}}=(\cf^{\kappa(l+1)}g)(x_{i,l\kappa+r}) 
		=g(\phi^{\kappa}(x_{i,r})) \\
		&\overset{\eqref{eq:analyticity_3}} = g(x_r) \overset{\eqref{eq:analyticity_4}} = g(\phi^{\kappa(l+1)}(x_r)) 
		=(\cf^{\kappa(l+1)}g)(x_r)\overset{\eqref{eq:analyticity_2}}=f(x_r),
		\end{align*}
		which yields \eqref{eq:analyticity_1}.
		
		(ii)$\Rightarrow$(iii).
		This implication is a consequence of \eqref{eq:analyticity_4} with $l=0$ and the following equality
		\begin{align*}
		\phi^\kappa (x_{i, l\kappa+r})
		= \begin{cases}
		x_{i,\kappa (l-1) + r} & \text{if } l\in\nbb \text{ and } r\in J_{[1,\kappa]}, \\
		x_r & \text{if } l=0 \text{ and } r\in J_{[1,\kappa-1]},
		\end{cases}
		\end{align*}
		which holds for all $l\in \zbb_+$ and $r\in J_{[1,\kappa]}$.
		
		(iii)$\Rightarrow$(i). It is obvious.
	\end{proof}
	
	Note that in the case of finite measure spaces an operator  $\cf\in\B[L^2(\mu)]$ with $(X,\ascr,\mu)$ and $\phi$ given by $\eqref{def:composition-operator}$ is not analytic.
	Indeed, setting $f\equiv 1$ we deduce from Lemma~\ref{lem:analyticity} that $f\in L^2(\mu)\cap\RngInf(\cf)$, which contradicts analyticity of $\cf$.
	Using the above idea we can characterize analyticity of bounded composition operators on directed graphs with one circuit in the following way.
	
	\begin{pro}
		\label{prop:analiticity}
		Suppose $\eqref{def:composition-operator}$ holds and $\cf\in\B[L^2(\mu)]$.
		Then the following are equivalent{\em :}
		\begin{enumerate}
			\item $\cf$ is analytic,
			\item $\sum_{i=1}^\eta \sum_{l=0}^{\infty} \mu(x_{i,l\kappa + r})$ is divergent for all $r\in J_{[1,\kappa]}$.
		\end{enumerate}
	\end{pro}
	\begin{proof}
		(i)$\Rightarrow$(ii). Suppose, to the
		contrary, that $\sum_{i=1}^\eta
		\sum_{l=0}^{\infty}
		\mu(x_{i,l\kappa+r})<\infty$ for some $r\in J_{[1,\kappa]}$. Set
		\begin{align}
		\label{deffr}
		f_r(x)=
		\begin{cases}
		1 & \text{if } x\in\{x_r\}\cup\{x_{i,l\kappa+r}\colon i\in J_{[1,\eta]},\;l\in \zbb_+\}, \\
		0 & \text{otherwise,}
		\end{cases} \quad x\in X.
		\end{align}
		Then, in view of Lemma~\ref{lem:analyticity}, function $f\in L^2(\mu)\cap\RngInf(\cf)$, which contradicts analyticity of $\cf$.
		
		(ii)$\Rightarrow$(i). Assume that (ii)
		holds and $f\in \RngInf(\cf)$. Then, using
		Lemma~\ref{lem:analyticity}(ii) we can represent $f$ as 
		$f=\sum_{r=1}^\kappa f(x_r)f_r$, where
		$f_r$ is given by \eqref{deffr} for every $r\in J_{[1,\kappa]}$. This, the
		facts that the supports of functions $f_r$,
		$r\in J_{[1,\kappa]}$ are disjoint and $f\in L^2(\mu)$
		together with (ii)
		imply that $f\equiv0$, which completes the
		proof.
	\end{proof}
	
	The following important result is a direct consequence of Theorem~\ref{th:m-iso-branches} and Proposition~\ref{prop:analiticity}.
	
	\begin{cor}
		\label{cor:analytic}
		If $\eqref{def:composition-operator}$ holds and $\cf\in\B[L^2(\mu)]$ is an $m$--isometry for $m\ge 2$, then $\cf$ is analytic.
	\end{cor}
	
	In what follows we investigate whether the Cauchy dual operator of a composition operator on a directed graph with on circuit is subnormal.
	The notion of the Cauchy dual operator was introduced by Shimorin in \cite{shimorin-2001} on the occasion of study of the Wold-type decomposition and the wandering subspace property.
	Recently, these topics have been very popular and authors established interesting results (see e.g. \cite{A-C-T-2018,A-C-J-S-2017,B-S-2018}).
	In particular, there is still an open problem of determining a characterization of subnormality of the Cauchy dual operator of a $2$--isometry.
	Anand et al. provided two sufficient conditions under which $2$--isometric operator satisfies the above property, namely, the kernel condition (see \cite[Theroem 3.3]{A-C-J-S-2017}) and $\Delta_T$--regularity (see \cite[Theorem~4.5]{A-C-J-S-2017}).
	In what follows, we also characterize when composition operators on directed graphs with one circuit satisfy these properties.

	Let us recall that for a left-invertible
	operator $T\in\B$ the Cauchy dual operator $T'$ of $T$ is given by $T'=T(T^*T)^{-1}$.
	It is well-known (see~\cite{singh-1974})  that, if $\cf\in\B[L^2(\mu)]$ is a composition operator, then 
	\begin{align}
	\label{cauchy-composition}
	\cf^*\cf f = \hn{} f, \quad f\in L^2(\mu),
	\end{align}
	 and consequently, if moreover $\cf$ is left-invertible, then $\cf'$ is a weighted composition operator with symbol $\phi$ and weight $w_\phi:=\frac{1}{\hn{}\circ \phi}$, i.e.,
	\begin{align}
	\label{eq:cauchy-dual-composition}
	\cf' f = \frac{1}{\hn{} \circ \phi} \cdot \cf f, \quad f\in L^2(\mu).
	\end{align}
	
	In what follows we use the following notation.
	If $(X,\ascr,\mu)$ is a discrete measure space and $w\colon X\to(0,\infty)$ is a function, then by $\mu_w$ and $\{\hat{w}_n\}_{n=0}^\infty$ we denote a discrete measure on $\ascr$ and a sequence of functions uniquely determined by 
	\begin{align}
	\label{muw}
	\mu_w(x)=|w(x)|^2\mu(x),\quad x\in X
	\end{align}
	and 
	\begin{align}
	\label{sdvjoisdvo}
	\hat{w}_0\equiv 1, \quad \hat{w}_{n+1}=\prod_{j=0}^n w\circ\phi^j, \quad  n\in\zbb_+.
	\end{align}
	By $\h{\phi,w}$ we denote the Radon-Nikodym derivative $\frac{d\mu_w\circ\phi^{-1}}{d\mu}$.
	
	We gather below some necessary notation and properties of the Cauchy dual operator $\cf'$ of $\cf$, which are used later.
	
	\begin{lem}
		\label{lem:cauchy-dual-2-iso}
		If $\eqref{def:composition-operator}$ holds with $\kappa=1$, $\cf\in\B[L^2(\mu)]$ is a $2$--isometry and $w:=w_\phi$, then
		\begin{enumerate}
			\item $\hat{w}_n(x)=\begin{cases}
			\alpha^{n} & \text{if } x\in \{x_1\}\cup\{x_{i,1}\colon i\in J_{[1,\eta]}\},\\
			\alpha^{(n+1-j)} & \text{if } x\in\{x_{i,j}\colon i\in J_{[1,\eta]},\; j\in J_{[2,n]}\},\\
			1 & \text{otherwise},
			\end{cases} \; x\in X,\; n\in\zbb_+$
			\item $\h{\phi^n,\hat{w}_n}(x)=\begin{cases}
			\frac{\alpha^{2n}\mu(x_1)+c\sum_{j=1}^n   \alpha^{2(n+1-j)}}{\mu(x_{1})} & \text{if } x=x_1, \\
			1 & \text{otherwise},
			\end{cases} \;\; x\in X,\; n\in\zbb_+$,
		\end{enumerate}
		where $\alpha=\frac{\mu(x_1)}{\mu(x_1)+c}$
		and $c=\sum_{i=1}^\eta\mu(x_{i,1})$.
	\end{lem}
	\begin{proof}
		It follows from Corollary~\ref{cor:m-iso-characterization-one-element} that there exists a system $\{c_i\}_{i=1}^\eta\subseteq(0,\infty)$ such that
		\begin{align}
		\label{eq:mux} \mu(x_{i,j})=c_i \quad
		i\in J_{[1,\eta]}, \; j\in\nbb.
		\end{align}
		
		(i) A direct computation shows that (i) holds for $n=0,1$.
		Using \eqref{eq:mux}, for $k\in\nbb$ and $x\in X$ we get
		\begin{align*}
		{w}_\phi\circ \phi^k(x)=\begin{cases}
		\alpha & \hbox{if } x\in \{x_1\}\cup \{x_{i,j}\colon i\in J_{[1,\eta]},\; j\in J_{[2,k+1]}\},\\
		1 & \hbox{otherwise}. \\
		\end{cases}
		\end{align*}
		This and induction argument show that (i) holds for $n\ge 2$.
		
		(ii) Let us recall that, by \cite[eq.\ (6.5)]{B-J-J-S-2018},
		\begin{align*}
		\h{\phi^n,\hat{w}_n}(x)=\frac{\mu_{\hat{w}_n}((\phi^n)^{-1}(\{x\}))}{\mu(x)}, \quad x\in X.
		\end{align*}
		This implies that
		\begin{align*}
		\h{\phi^n,\hat{w}_n}(x_{i,j}) &=\frac{\mu_{\hat{w}_n}(x_{i,j+n})}{\mu(x_{i,j})}\\
		&\overset{\eqref{muw}}{=}
		\frac{\hat{w}_n(x_{i,j+n})^2\mu(x_{i,j+n})}{\mu(x_{i,j})}
		\overset{\mathrm{(i)\&\eqref{eq:mux}}}{=}1, \quad
		i\in J_{[1,\eta]},\; j\in \nbb.
		\end{align*}
		Similarly,
		\begin{align*}
		\h{\phi^n,\hat{w}_n}(x_{1})
		&=\frac{\mu_{\hat{w}_n}(\{x_1\}\cup\{x_{i,j} \colon i\in J_{[1,\eta]}, j=1,\ldots,n\})}{\mu(x_{1})} \\
		&\overset{\eqref{muw}}{=} \frac{\hat{w}_n(x_1)^2\mu(x_1)+\sum_{i\in J_{[1,\eta]}}\sum_{j=1}^n \hat{w}_n(x_{i,j})^2\mu(x_{i,j})}{\mu(x_{1})} \\
		&\overset{\mathrm{\eqref{eq:mux}}}{=} \frac{\alpha^{2n}\mu(x_1)+\Big(\sum_{i\in J_{[1,\eta]}}c_i\Big)\sum_{j=1}^n \alpha^{2(n+1-j)}}{\mu(x_{1})} \\
		&\overset{\mathrm{(i)}}{=} \frac{\alpha^{2n}\mu(x_1)+c\sum_{j=1}^n \alpha^{2(n+1-j)}}{\mu(x_{1})}
		\end{align*}
		which completes the proof.
	\end{proof}

	Now we provide a characterization of subnormality of the Cauchy dual operator of a composition operator on a directed graph with one circuit.
	
	\begin{thm}
		\label{th:subnormal-characterization}
		If $\eqref{def:composition-operator}$ holds  and $\cf\in\B[L^2(\mu)]$, then the following conditions are equivalent{\em :}
		\begin{enumerate}
			\item $\cf'$ is subnormal,
			\item $\{\h{\phi^n,\hat{w}_n}(x)\}_{n=0}^\infty$ is a Stieltjes moment sequence for all $x\in X$, where $w=w_\phi$ and $\{\hat{w}_n\}_{n=0}^\infty$ is given by \eqref{sdvjoisdvo}.
		\end{enumerate}
	\end{thm}
	\begin{proof}
	The proof follows directly from \eqref{eq:cauchy-dual-composition} and  \cite[Theorem~49]{B-J-J-S-2018}.
	\end{proof}
	
	The below theorem is the main result of this section.
	It answers affirmatively the question whether the Cauchy dual operator of $2$--isometric composition operator on a directed graph with one circuit with $\kappa=1$ is subnormal.
	
	\begin{thm}
		\label{th:subnormal}
		If $\eqref{def:composition-operator}$ holds with $\kappa=1$ and $\cf\in\B[L^2(\mu)]$ is a $2$--isometry, then $\cf'$ is subnormal.
	\end{thm}
	\begin{proof}
		By Lemma~\ref{lem:cauchy-dual-2-iso} and Theorem~\ref{th:subnormal-characterization}, it remains to prove that sequence
		$\big\{\frac{\alpha^{2n}\mu(x_1)+c\sum_{j=1}^n \alpha^{2(n+1-j)}}{\mu(x_{1})}\big\}_{n=0}^\infty$
		is a Stieltjes moment sequence, where $\alpha$ and $c$ are as in Lemma~\ref{lem:cauchy-dual-2-iso}.
		Observe that
		\begin{align*}
		\frac{\alpha^{2n}\mu(x_1)+c\sum_{j=1}^n \alpha^{2(n+1-j)}}{\mu(x_{1})}
		&=\alpha^{2n}+\frac{c\alpha^2}{\mu(x_1)}\frac{\alpha^{2n}-1} {\alpha^{2}-1} \\
		&=\frac{\mu(x_1)+c}{2\mu(x_1)+c} \alpha^{2n} + \frac{\mu(x_1)}{2\mu(x_1)+c}, \quad n\in \zbb_+.
		\end{align*}
		Hence
		$\big\{\frac{\alpha^{2n}\mu(x_1)+c\sum_{j=1}^n \alpha^{2(n+1-j)}}{\mu(x_{1})}\big\}_{n=0}^\infty$
		is a Stieltjes moment sequence with the representing measure
		$\frac{\mu(x_1)+c}{2\mu(x_1)+c}\delta_{\alpha^{2}}+\frac{\mu(x_1)}{2\mu(x_1)+c}\delta_{0}$, where $\delta_x$ stands for the Borel probability measure on $\rbb$ supported on $\{x\}$ for $x\in\rbb$.
		This completes the proof.
	\end{proof}
	
	Suppose \eqref{def:composition-operator} holds.
	In what follows we denote by  $\chi_i$ the
	characteristic function of $\{x_i\}$ for $i\in
	J_{[1,\kappa]}$ and by $\chi_{i,j}$ the characteristic function of $\{x_{i,j}\}$ for $i\in J_{[1,\eta]}$ and $j\in\nbb$.
	If $A\subseteq X$, then $\chi_A$ stands for the characteristic function of $A$.

	Let us recall that a $2$--isometric operator $T\in\B$ is said to be $\Delta_T$--regular if
	$\Delta_T T = \Delta_T^\frac{1}{2} T \Delta_T^\frac{1}{2}$, where $\Delta_T = T^*T -I$ (see \cite{A-C-J-S-2017,B-S-2018}).
	By \cite[Lemma~1(a)]{richter-1988} operator $\Delta_T$ is positive, hence the above definition makes sense.
	Now we establish an equivalent condition for a $2$--isometric composition operator on a directed graph with one circuit to be $\Delta_{T}$--regular.
	
	\begin{pro}
		\label{pro:regular}
		Assume \eqref{def:composition-operator} holds and $\cf$ is $2$--isometric.
		Then $\cf$ is $\Delta_{\cf}$--regular if and only if $\kappa=1$.
	\end{pro}
	\begin{proof}
		Assume $\kappa=1$.
		We want to show that $\Delta_{\cf} \cf= \Delta_{\cf}^\frac{1}{2} \cf \Delta_{\cf}^\frac{1}{2}$.
		By Corollary~\ref{cor:m-iso-characterization-one-element} there exists $\{c_i\}_{i=1}^\eta\subseteq(0,\infty)$ such that
		\begin{align*}
		\mu(x_{i,j})=c_i \quad i\in
		J_{[1,\eta]}, \; j\in\nbb.
		\end{align*}
		It follows from \eqref{cauchy-composition} that 
		\begin{align}
			\label{sdferff}
			\Delta_{\cf}f = (\h\phi - 1)f, \quad f\in L^2(\mu).
		\end{align} 
		Moreover, by \eqref{rndx151}
		\begin{align*}
		(\h\phi-1)(x) = \begin{cases}
		\frac{c}{\mu(x_1)} & \text{if } x = x_1, \\
		0 & \text{otherwise,}
		\end{cases}
		\end{align*}
		where $c=\sum_{i=1}^\eta c_i$.
		This implies that $\Delta_{\cf}f=\frac{c}{\mu(x_1)} f(x_1)\chi_1$ and consequently $\Delta_{\cf}^\frac{1}{2}f=\sqrt{\frac{c}{\mu(x_1)}} f(x_1)\chi_1$ for $f\in L^2(\mu)$.
		Hence $\sqrt{\frac{c}{\mu(x_1)}}\chi_1 = \Delta_{\cf}^\frac{1}{2}\chi_1 = \Delta_{\cf}^\frac{1}{2}(\chi_1 + \sum_{i=1}^{\eta} \chi_{i,1})$.
		These together yield
		\begin{align*}
		\Delta_{\cf} \cf f &= \frac{c}{\mu(x_1)} f(x_1) \chi_1
		=\sqrt{\frac{c}{\mu(x_1)}} f(x_1) \Delta_{\cf}^\frac{1}{2}  \Big(  \chi_1 + \sum_{i=1}^{\eta} \chi_{i,1}\Big) \\
		&=\sqrt{\frac{c}{\mu(x_1)}} f(x_1) \Big(\Delta_{\cf}^\frac{1}{2} \cf \Big)  \chi_1 = \Delta_{\cf}^\frac{1}{2} \cf\Delta_{\cf}^\frac{1}{2}f, \quad f\in L^2(\mu).
		\end{align*}
		Therefore $\cf$ is $\Delta_{\cf}$--regular.
		
		Now, suppose $\kappa>2$.
		Assume to the contrary that $\cf$ is $\Delta_{\cf}$--regular.
		It is easily seen that
		\begin{align}
		\label{sdvf5rr}
		 \Delta_{\cf} \cf \chi_{1} = \Delta_{\cf} \chi_{2} \overset{\eqref{sdferff}}= (\h\phi (x_{2}) - 1)\chi_{2}
		\end{align}
		and
		\begin{align}
		\Delta_{\cf}^\frac{1}{2} \cf \Delta_{\cf}^\frac{1}{2} \chi_{1} &= \Delta_{\cf}^\frac{1}{2}\cf (\h\phi(x_{1}) - 1)^\frac{1}{2}\chi_{1} \notag \\ \label{sxfdsss}
		&= (\h\phi(x_{1}) - 1)^\frac{1}{2}(\h\phi (x_{2}) - 1)^\frac{1}{2}\chi_{2}.
		\end{align}
		Observe that, by \eqref{eq:3-iso-polynomial} from Theorem~\ref{th:characterization_3_isom}, $\h\phi (x_{2}) - 1 = \frac{\frac{c}{\kappa}}{\frac{c}{\kappa}(2-\kappa)+t}$ and $\h\phi (x_{1}) - 1 = \frac{\frac{c}{\kappa}}{\frac{c}{\kappa}(1-\kappa)+t}$.
		Since $\cf$ is $\Delta_{\cf}$--regular, this, \eqref{sdvf5rr} and \eqref{sxfdsss}  imply that
		\begin{align*}
		\h\phi (x_{2}) - 1 = \h\phi (x_{1}) - 1
		\end{align*}
		which yields a contradiction.
		
		Finally, if $\kappa=2$, then it is a matter of a similar verification as in the case when $\kappa>2$ to prove that $\cf$ is not $\Delta_{\cf}$--regular  (take $\chi_1$ and use $\Delta_{\cf}$--regularity to get a contradiction by showing that $\mu(x_2)<0$).
		This completes the proof.
	\end{proof}
	
	If \eqref{def:composition-operator} holds with $\kappa=1$ and $\cf\in\B[L^2(\mu)]$ is $2$--isometric, then a combination of Proposition~\ref{pro:regular} with \cite[Theorem~4.5]{A-C-J-S-2017} implies that $\cf'$ is subnormal.
	The above argument seems to be simpler than the proof of Theorem~\ref{th:subnormal}, however, this approach does not yield a direct form of the representing measure.
	
	Recall that $T\in\B$ is said to satisfy the
	kernel condition if $T^*T\Ker(T^*) \subseteq
	\Ker(T^*)$ (see \cite{A-C-J-S-2017}). Observe
	also that if $X$ is a set, $\phi\colon X\to X$
	and $f\colon X\to \C$ is a function such that $f$
	is constant on $\phi^{-1}(\{x\})$ for all $x\in
	X$, then the function $f\circ \phi^{-1}\colon
	X\to \C$, where
	\begin{align*}
	(f\circ \phi^{-1})(x):=
	\begin{cases}
	f(y) & \text{if there exists $y \in X$ such that }x=\phi(y), \\
	0 & \text{otherwise.}
	\end{cases}
	\end{align*}
	is well-defined and $(f\circ \phi^{-1})\circ\phi
	= f$. This implies that $f$ is constant on
	$\phi^{-1}(\{x\})$ for all $x\in X$ if and only
	if there exists a function $g\colon X\to \C$ such
	that $g\circ \phi =f$. The following proposition
	provides a characterization of left-invertible
	composition operators that satisfy the kernel
	condition.
	
	\begin{pro}
		\label{pro:composition-kernel-condition}
		Let $(X,\ascr,\mu)$ be a discrete measure space,
		$\phi$ be a nonsingular self-map of $X$ and
		$\cf\in\B[L^2(\mu)]$ be a composition operator.
		Then the following conditions hold{\em :}
		\begin{enumerate}
			\item $f\in\Rng(\cf)$ if and only if
			$f\colon X\to \C$ is a function such that $f$ is
			constant on $\phi^{-1}(\{x\})$ for all $x\in X$
			and $f\circ \phi^{-1}\in L^2(\mu)$,
			\item if $\cf$ left-invertible, then
			$\cf$ satisfies the kernel condition if and only
			if $\hn{}$ is constant on preimages
			$\phi^{-1}(\{x\})$ for all $x\in X$.
		\end{enumerate}
	\end{pro}
	\begin{proof}
		(i) To prove the ``if'' part, assume that
		$f\in\Rng(\cf)$. Then there exists $g \in
		L^2(\mu)$ such that $f=g\circ \phi$ and hence $f$
		is constant on $\phi^{-1}(\{x\})$ for all $x\in
		X$. Since $g \in L^2(\mu)$, $g|_{\phi(X)}=(f\circ
		\phi^{-1})|_{\phi(X)}$ and $g|_{X\setminus
			\phi(X)}=0$, we deduce that $f\circ \phi^{-1}\in
		L^2(\mu)$. It is a routine matter to show that
		the reverse implication is also true.
		
		(ii) Since $\cf$ is left-invertible, the range
		$\Rng(\cf)$ is closed. Hence, it is a direct
		consequence of the kernel-range decomposition
		that $\Ker(\cf^*)^\perp=\Rng(\cf)$. Therefore,
		$\cf$ satisfies the kernel condition if and only
		if
		\begin{align}
		\label{eq:kernel-condition} \cf^*\cf f \in
		\Rng(\cf), \quad f\in\Rng(\cf).
		\end{align}
		
		Now, suppose that $\cf$ satisfies the kernel
		condition and fix $x\in X$. Let $y,z\in X$ by
		such that $\phi(y)=\phi(z)=x$ and set
		$f=\cf\chi_{\{x\}} = \chi_{\phi^{-1}(\{x\})}$. It
		is easily seen that $f(y)=f(z)=1$. This combined
		with the fact that $\cf^*\cf f= \hn{}f$,
		\eqref{eq:kernel-condition} and (i), imply that
		$\hn{}(y)=\hn{}(z)$. Thus $\hn{}$ is constant on
		preimages $\phi^{-1}(\{x\})$ for all $x\in X$.
		
		To prove the reverse implication suppose that
		$\hn{}$ is constant on preimages
		$\phi^{-1}(\{x\})$ for all $x\in X$. Then the
		function $\hn{}\circ \phi^{-1}$ is well-defined
		and bounded. Thus, if $f\in\Rng(\cf)$, then, by
		(i), we get
		\begin{align*}
		\cf^*\cf f =\hn{}f=\big((\hn{}\circ\phi^{-1})( f
		\circ\phi^{-1})\big)\circ \phi \in \Rng(\cf),
		\quad f\in\Rng(\cf),
		\end{align*}
		which completes the proof.
	\end{proof}
	
	It is known that the Cauchy dual operators of $2$--isometries that satisfy the kernel condition are subnormal (see \cite[Theorem~3.3]{A-C-J-S-2017}).
	The following result states that if $\cf$ is a $2$--isometric composition operators on a directed graph with one circuit, then $\cf$ does not satisfy the kernel condition.
	
	\begin{thm}
		\label{th:kernel-condition}
		Suppose \eqref{def:composition-operator} holds and $\cf\in\B[L^2(\mu)]$ is a $2$--isometry.
		Then $\cf$ does not satisfy the kernel condition.
	\end{thm}
	\begin{proof}
		Suppose $\kappa=1$. Since $\cf$ is a
		$2$--isometry, it follows from
		Corollary~\ref{cor:m-iso-characterization-one-element}
		that there exists $\{c_i\}_{i=1}^\eta\subseteq(0,\infty)$ such that
		$\mu(x_{i,j})=c_i$ for $i\in J_{[1,\eta]}$
		and $j\in\nbb$. Define $c=\sum_{i=1}^\eta
		c_i$ and observe that $c<\infty$ due to
		boundedness of $\cf$. Note that $\hn{}(x_1)
		= 1 + \frac{c}{\mu(x_1)}$ and
		$\hn{}(x_{i,1})=1$ for $i\in J_{[1,\eta]}$.
		It follows from
		Proposition~\ref{pro:composition-kernel-condition}
		that $\cf$ satisfies the kernel condition
		if and only if $\hn{}(x_1)=\hn{}(x_{i,1})$
		for $i\in J_{[1,\eta]}$. The latter implies
		that $c=0$, which contradicts our
		assumptions.
		
		Now assume $\kappa>1$. It follows from
		Theorem~\ref{th:characterization_3_isom} that
		there exist $t\in(0,\infty)$, 
		$\{c_i\}_{i=1}^\eta\subseteq(0,\infty)$
		and a polynomial $w$ in one indeterminate $x$
		with real coefficients of degree at most $1$ such
		that
		$\mu(x_{i,j})=c_i$ for $i\in J_{[1,\eta]}$,
		$j\in\nbb$ and $\mu(x_i)=w(i)$ for $i\in
		J_{[1,\kappa]}$, where
		$w(x)=\frac{c}{\kappa}x-c+t$ and
		$c=\sum_{i=1}^\eta c_i <\infty$. Note that
		$\hn{}(x_{1}) = \frac{w(2)}{w(1)}$ and
		$\hn{}(x_{i,1})=1$ for $i\in J_{[1,\eta]}$.
		If $\cf$ satisfies the kernel condition,
		then it follows from
		Proposition~\ref{pro:composition-kernel-condition}
		that $\hn{}(x_{1}) = \hn{}(x_{i,1})$, for
		$i\in J_{[1,\eta]}$. Again, the last
		equality implies that $c=0$, which
		contradicts our assumptions. Hence, the
		proof is completed.
	\end{proof}

\end{document}